\documentclass[11pt]{amsart}

\usepackage{amssymb, graphicx}
\usepackage{graphicx}
\numberwithin{equation}{section}
\usepackage{topcapt}
\usepackage[all]{xy}
\usepackage{color}
\usepackage{amsfonts}
\usepackage{amsmath}
\usepackage{latexsym}
\usepackage{amscd}
\usepackage{verbatim}
 \def\pd#1#2{\dfrac{\partial#1}{\partial#2}}
\usepackage{color}

\newcommand{\mf}{\mathfrak}
\newcommand{\g}{\mf{g}}\newcommand{\cac}{\delta}
\newcommand{\h}{\mf{h}}

\allowdisplaybreaks[1]

\newcommand{\pip}{t}
\newtheorem{theorem}{Theorem}[section]
\newtheorem{proposition}[theorem]{Proposition}
\newtheorem{lemma}[theorem]{Lemma}
\newtheorem{corollary}[theorem]{Corollary}
 
\theoremstyle{remark}
\newtheorem{remark}[theorem]{Remark}
\newtheorem{definition}[theorem]{Definition}

\newcommand{\A}{\mathcal{A_\mathfrak{R}}}

\begin{document}
  \title{The adjoint representation inside the exterior algebra of a simple Lie algebra}
  \author{Corrado De Concini\\Paolo Papi\\Claudio Procesi}

  \keywords{Exterior algebra, invariants, transgression.}
  \subjclass[2010]{17B20}
\maketitle

\begin{abstract} For a simple complex Lie algebra $\g$ we study the space of invariants $A=\left( \bigwedge \mathfrak g^*\otimes\mathfrak g^*\right)^{\mathfrak g}$, which describes the isotypic component of type $\g$ in $ \bigwedge \mathfrak g^*$, as a module over the algebra of invariants $\left(\bigwedge \mathfrak g^*\right)^{\mathfrak g}$.  As main result  we prove that $A$ is a  free module, of rank twice the rank of $\g$, over the exterior algebra generated by all primitive invariants in $(\bigwedge \mathfrak g^*)^{\mathfrak g}$,  with the exception of the one of highest degree.
\end{abstract}
\section{Introduction}
  Let $\mathfrak g$ be a simple Lie algebra (over $\mathbb C$) of dimension $n$ and rank $r$. Fix a Cartan subalgebra $\h$ in $\g$. Let $\Delta$ be the corresponding root system, $W$ the Weyl group, $\Delta^+$ a positive system.
 We use  the Killing form $(\cdot, \cdot)$  to identify $\g$ and $\g^*$ when convenient.\par
The exterior algebra $\bigwedge \mathfrak g$ has been extensively studied as representation of $\mathfrak g$: see e.g. \cite{Kostant}, \cite{K}, \cite{R}, \cite{B}. 
The invariant algebra  $\Gamma=(\bigwedge \mathfrak g^*)^{\mathfrak g}$  is the  cohomology of $\g$ and it is an exterior algebra  $\bigwedge(P_1,\ldots,P_r)$ over primitive generators $P_i$ of degree  $2m_i+1$, where the integers $m_i,$ with $  m_1\leq\ldots\leq m_r$ are the {\em exponents} of $\Delta$, \cite{Bou}.

Among the other isotypic components, of particular interest is the component of type $\mathfrak g$, which  is completely described as: \begin{equation}
\label{inv} A:= \left( \bigwedge \mathfrak g^*\otimes\mathfrak g^*\right)^{\mathfrak g}=\hom(\bigwedge \mathfrak g,\mathfrak g^* )^{\mathfrak g} =\hom(\mathfrak g ,\bigwedge \mathfrak g^*)^{\mathfrak g}. 
\end{equation} So $A$ is the space of  multilinear alternating functions from   $\g$ to $\g^*$ which are  $\mathfrak g$-equivariant.
\smallskip

By the work of Kostant \cite{K}, it is known that $\dim(A)= 2^rr$. Notice that the $\mathbb Z/2\mathbb Z$ grading of $\bigwedge\g^*$  allows to define a Lie superalgebra structure on $A$.

The Poincar\'e polynomial $GM (q)$  describing the dimension of $A$ in each degree is given by a formula conjectured by Joseph and proved by  Bazlov  \cite{B}:
\begin{equation}
\label{ps}GM (q)=(1+q^{-1}) \prod_{
i=1}^{r-1}
(1+q^{2m_i+1})\sum_{
i=1}^r
q^{2m_i}.  \end{equation} 
Clearly $A$ is a (left or right) module over the  exterior algebra $\bigwedge(P_1,\ldots,P_{r }).$
Writing the right hand side  of \eqref{ps} as $\prod\limits_{
i=1}^{r-1}(1+q^{2m_i+1})
\sum\limits_{
i=1}^r[
q^{2m_i}+q^{2m_i-1}]$ suggests that there might exist  
 elements $f_i,u_i\in A$ of degrees $2m_i,2m_i-1$ which generate $A$ as a free module over $\bigwedge(P_1,\ldots,P_{r-1}).$ This is indeed our main result. 
\subsection{Description of the results}
The Killing form of $\g$ induces
on $A$  a graded symmetric form which will be denoted $e(a,b)$, with values in $\Gamma$ (see \eqref{formae}). With the elements $f_i,u_i\in A$ as in Definition \ref{deffu} we have:
 
\begin{theorem}\label{main}
  $A$ is a free module, with basis the elements $f_i,u_i, i=1,\ldots,r,$ over the exterior algebra $\bigwedge(P_1,\ldots,P_{r-1}).$
\end{theorem} The main tool needed to prove this Theorem is a description of the form $e(-,-)$  on the proposed generators $f_i,u_i$, which is given by the following result. Recall that the primitive invariants are uniquely defined up to a non zero constant if the exponents are distinct. The latter condition is always verified unless
  $\g$ is of type $D_{2n}$; in this case $m_n=m_{n+1}=2n-1$.
\begin{proposition}\label{main1} Assume that the exponents of $\g$ are distinct. Then the two submodules  spanned by the $f_i$ and by  the $u_j$ are isotropic for the form  $e$. 
For each pair $i,j$ there exists a  non-zero rational constant  $c_{i,j}$ such that
\begin{equation}\label{forrmae}
e(f_i,u_j)=e(f_j,u_i)=\begin{cases}
c_{i,j}P_k\quad&\text{if}\quad m_i+m_j-1=m_k\quad \text{is an exponent,}\\
0\quad &\text{otherwise.}
\end{cases}
\end{equation}
\end{proposition}
\begin{proposition}\label{mainD} If 
 $\g$  is of type $D_{2n}$, the statements of  Proposition \ref{main1} hold true unless $m_i+m_j-1=2n-1$ or $\{i,j\}=\{n,n+1\}$. In this case we can choose $P_n$ and $P_{n+1}$ in such a way that
\begin{equation}\label{forrmaeD}
e(f_i,u_j)=e(f_j,u_i)=\begin{cases}
c_{i,j}(P_n+P_{n+1})\quad&\text{if $ i,j\not\in\{n,n+1\}$},\\
c_{i,j}P_n\quad&\text{if $ i=n$ and $j=1$},\\
c_{i,j}P_{n+1}\quad&\text{if $ i=n+1$ and $j=1$}.\\
\end{cases}
\end{equation}
and $e(f_n,u_{n+1})=c_{n,n+1}P_{2n}$, $e(f_{n+1},u_{n})=c_{n+1,n}P_{2n}$ while $e(f_n,u_n)=e(f_{n+1},u_{n+1})=0$.
\end{proposition}
 Finally, the full module structure of $A$ under the algebra of invariants is completed by the formulas for the multiplication of our basis by the last primitive element $P_r$. \begin{theorem}\label{main2} Set $c_i:=c_{i,r-i+1}$. The multiplication by $P_r$ is self adjoint for the form $e$. It is  given by the formulas
\begin{align}\label{p1}
e(f_i,u_{r-i+1})\wedge f_i&= -\sum_{j=1,\ j\neq i}^{r} c_ic_j^{-1}e(f_i,u_{r-j+1})f_j,\qquad  i=1,\ldots,r,\\\label{p2}
e(f_i,u_{r-i+1})\wedge u_i&= -\sum_{j=1,\ j\neq i}^{r} c_ic_j^{-1} e(f_i,u_{r-j+1})u_j ,\qquad i=1,\ldots,r.
\end{align} 
\end{theorem}
The constants $c_{i,j}$ depend on the choice of the primitive invariants.  We will exhibit,  in the course of the paper,  suitable choices of the $P_i$ as to normalize as much as possible these constants, which will be then   explicitly computed.\par
Clearly our Theorem \ref{main}  implies    formula \eqref{ps}, affording a proof very different from Bazlov's one.\par
Theorem \ref{main} has a Clifford counterpart due to Kostant \cite{K}: $A$ is free over the whole algebra of invariants $\Gamma$, generated under Clifford multiplication by $f_1,\ldots,f_r$. Kostant's theorem  is based on two main ingredients: the simplicity of the Clifford algebra and the the machinery of Chevalley's transgression. In our situation  we cannot use the former fact. 

However Chevalley transgression   allows us to reduce the computation of the bilinear form $e$ to that of another 
$S(\g)^\g$-valued bilinear form  on $S(\g)^\g$, and in turn, via Chevalley's restriction Theorem, to a $\mathbb C[\h]^W$-valued bilinear form on $\mathbb C[\h]^W$. 

This latter form can be introduced  in the following general framework. Let $G$ be a compact Lie group and $V$ an orthogonal representation. Given two invariants $a,b$  of degrees $h,k$  we obtain  a new invariant $a\circ b$ of degree $h+k-2$ by the formula $a\circ b:=(da,db)$  where $da, db$ are the differentials and $(da,db)$ is computed via the given  scalar product.
In the case of the adjoint representation, this pairing has been studied by Givental in \cite{Given} and Saito \cite{Saito} (we thank M. Rais and E. Vinberg for pointing out to us these references) and we are going to use Givental's results at least for the exceptional groups.
\smallskip

Let us denote by $R$  the ring of $G$--invariant polynomial functions on $V$,  and say that a homogeneous invariant $a$ is a {\em  generator} if it does not belong to $(R^+)^2$, i.e. it is not a product of invariants of positive degree.  Notice furthermore that, by the Leibnitz rule,  the pairing $a\circ b$  induces a similar composition in the vector space  $M:=R^+/(R^+)^2$. We will prove the following theorem.
\begin{theorem}\label{gendi}
Let $V$ be the reflection representation of an irreducible Weyl group. If $a,b$ are homogeneous generators   of the invariants of degrees $i,j$ and in degree $i+j-2$ there is a generator, then $a\circ b$ is a generator, except the special case of type $D_{2n}$ in degree $2n$. In this case the two generators $a,b$ of degree $2n$ can be chosen in such a way that $a\circ a=\ b\circ b=0$  modulo squares  while $a\circ b=0$ is a generator.
\end{theorem}

The above theorem will be proved in Proposition \ref{lap}  if $i+j\geq h$, $h$ being the Coxeter number of $\g$. In fact this will suffice to deduce  Theorem \ref{main}.

In the remaining cases the result will follow from a case by case analysis and will represent a key step in the proof of  Theorem \ref{main2}.  

Finally,   Theorem \ref{gendi} has another interesting application. If we allow the operation $a\circ b$ in the construction of invariants, we shall see that in each case, besides  the quadratic invariant it is enough to add either one   (cases $A_n,C_n,G_2$) or two more generators.  In this way we shall compute, using a computer,  the invariants in the case of $E_8$.
\begin{remark}\label{err}
\label{Led}  When all degrees of generators are distinct, fixing generators $\psi_1,\ldots,\psi_r$ we  have constants $d_{i,j}$ such that   $\psi_i\circ \psi_j=d_{i,j}\psi_k\mod (R^+)^2 $ when $\psi_k$ has the same degree as $\psi_i\circ\psi_j$. Using transgression, we associate to the $\psi_i$ primitive invariants 
$P_i$,  which in turn define the constants $c_{ij}$ via formula \eqref{forrmae}.
We will show in \eqref{latrai} that indeed, for any $i,j$, we have
$$c_{i,j}=\frac{d_{i,j}}{m_i+m_j}.$$
These constants will be computed explicitly for the classical groups in \S\ref{ddd} and for the exceptional groups in Tables 2,3,4.
\end{remark}
\par
Some ideas and techniques used in this paper have been developed in \cite{DCPPM} and \cite{Dolce}: we refer the reader 
to  Section \ref{FR} for a brief outline of these results.

 \section{The main construction}
 \subsection{Setup\label{genefo}} Let $V$ be an $n$-dimensional vector space.
Recall that on the exterior algebra  $\bigwedge V$  we have an action of  elements $x\in V^*$ as derivations, denoted  by $i(x)$, which extends  the duality action on $V$:
$$i(x)(v_1\wedge\ldots\wedge v_k)=\sum_{i=1}^k (-1)^{i+1}x(v_i)\,v_1\wedge\ldots \wedge \widehat{v_i}\wedge\ldots\wedge v_k,$$
$v_1,\ldots,v_k\in V$.  This formula extends to a  contraction action of $\bigwedge V^*$ on $\bigwedge V$.  
Remark that, given $u\in \bigwedge V$,  the map $x\mapsto i(x)u$  can be thought of as the element $\sum\limits_{h=1}^n i(x^h)u\otimes x_h\in\bigwedge V\otimes V ,$ where $\{x_h\},\{ x^j\}$ are dual bases of $V,V^*$. 

Given an invariant $p\in (\bigwedge \mathfrak g^*)^{\mathfrak g}$, the map $x\mapsto i(x)p,\ x\in\mathfrak g$ is $\mathfrak g$ equivariant and therefore defines an element of $A$,  which is  represented by the tensor $\sum\limits_{h=1}^n i(x_h)p\otimes x^h$ (using dual bases  $\{x_h\},\{x^h\}$ for $\g,\g^*$).

Let    $\theta(a),\theta^*(a) $  denote the adjoint and coadjoint action of $a\in\g$ on $\g,\g^*$, respectively. These actions extend to the exterior algebras as derivations of degree 0. \par

 In $\bigwedge \mathfrak g^* $  we have the Koszul differential $\cac$  which makes it a differential graded algebra. In degree 1, the differential  $\cac:\mathfrak g^*\to\bigwedge^2\mathfrak g^*$  is dual    to the bracket map.  Explicitly
\begin{equation*}
\cac(u)= \frac1{ 2}\sum_{i=1}^n x^i\wedge \theta^*(x_i)u.
\end{equation*} 
\begin{remark}
Recall that on $ \bigwedge\g^*\otimes \g^*$ is defined the standard Koszul differential, which is the  sum of $-\cac  \otimes 1$ and of 
$ \psi \otimes a\mapsto \sum\limits_{i=1}^n   x^i\wedge\theta(x_i)(\psi \otimes a)$, where $\{x_i\}, \{x^j\}$ are as above dual bases of $\g,\ \g^*$ respectively.
If $\psi \otimes a$ is invariant   this second term vanishes and  the total differential is just $-\cac  \otimes 1$.
\end{remark}
\vskip10pt
Using the non degenerate Killing form one defines also the differential  (in homology):
$$\partial:=-\cac^t.$$  The  Killing form    identifies $\bigwedge \g^*$ with $\bigwedge \g$,  so we can apply, by duality, the two differentials $\cac$ and $\partial$ to $\bigwedge \g$. Under these identifications  the contraction $i$ is just the adjoint of the wedge multiplication  $\varepsilon(x):\bigwedge\g\to\bigwedge\g, \varepsilon(x)(y)=x\wedge y,\,x,y \in\bigwedge\g$. We have (see \cite[(89),(90)]{K})
\begin{equation}\label{forba}\delta=\frac{1}{2}\sum_{h=1}^n\varepsilon(x_h)\theta(x^h),\quad \partial=\frac{1}{2}\sum_{h=1}^n\theta(x^h)i(x_h).\end{equation}
For the space  $\Gamma= (\bigwedge \g^*)^\g=(\bigwedge \g)^\g$  of invariant forms   Kostant \cite[Proposition 22]{K}   proves:
\begin{equation}\label{Gamma}\Gamma=\ker\cac\cap\ker\partial=\ker L,\end{equation} where
\begin{equation}\label{lapl}L= \cac \partial+\partial \cac =\frac12\sum_{i=1}^n \theta(z_i)^2\end{equation}
is the Laplacian and $\{z_i\}$ is an orthonormal basis of $\g$ with respect to the Killing form.
Recall  the useful  formulas
 \begin{equation}\label{forbas}
i(x)\,\cac +\cac \,i(x)=\theta(x),\quad x\in\g.\end{equation} 
\begin{equation}\label{aus}i(\cac x) = - (\partial i(x)+i(x)\partial),\quad x\in\g.\end{equation}
Formula \eqref{forbas} follows from  \cite[(92)]{K}. To prove \eqref{aus}, 
take $u,v\in\bigwedge \g$; then $(i(\delta(x))(u),v)=(u,\delta(x)\wedge v)=(u,\delta(x\wedge v))+(u,x\wedge \delta(v))=-(\partial(u),x\wedge v))+(i(x)(u),\delta(v))=
-(i(x)(\partial(u)),v)-(\partial(i(x)(u)),v)$.

\begin{lemma}\label{conto} If $p\in \Gamma$, then 
\begin{enumerate}
\item $\cac  p=\partial p=0$.
\item $\cac i(x)p=0$.
\item $\partial i(x)p=-i(\cac x)p$.
\item $\cac i(\cac x)p=-\frac12 i(x)p.$
\end{enumerate}
\end{lemma}
\begin{proof} Relation (1) follows at once from \eqref{Gamma}.   To prove (2), compute using (1) and \eqref{forbas}: 
$$\cac i(x)p=(\cac i(x)+i(x)\cac )(p)=\theta(x)(p)=0. $$
Part (3) follows directly from (1) and \eqref{aus}. Finally, to prove (4), use (1),  \eqref{aus} and \eqref{lapl}:
$$
\cac i(\cac x)p=-\cac (\partial i(x)+i(x)\partial)(p)=-\cac \partial i(x)(p)=-(\cac \partial+\partial \cac )i(x)(p)=-\frac12 i(x)(p).
$$
\end{proof}
\subsection{The elements $f_i,u_i $\label{201}}We  now introduce  a   basic definition.  Let  $m:  \bigwedge \mathfrak g \otimes \mathfrak g\to  \bigwedge \mathfrak g$ be the multiplication map 
\begin{equation}\label{multiplication}m(\alpha\otimes b):=\alpha\wedge b.\end{equation} Choose primitive generators $P_1,\ldots,P_r$ for $\Gamma$ and dual bases $\{x_h\}, \{x^h\}\in\g$ with respect to  the Killing form.  
\begin{definition}\label{deffu} Set, for $i=1,\ldots,r$
\begin{equation}
\label{gliuf} f_i:=\frac{1}{\deg(P_i)}\sum_{h=1}^n i(x_h)P_i\otimes x^h,\quad u_i:=2(\partial\otimes 1 )f_i=\frac{2}{\deg(P_i)}\sum_{h=1}^n\partial i(x_h)P_i\otimes x^h.
\end{equation}
\end{definition}

\begin{lemma}
We have $m(f_i)=P_i,\quad (\cac\otimes 1) u_i= f_i,\ (\cac\otimes 1) f_i=0$.
\end{lemma}
\begin{proof} The first statement is a consequence of the chosen normalization:
$$m(f_i)=\frac{1}{\deg(P_i)}\sum_{h=1}^n i(x_h)P_i\wedge x^h=\frac{1}{\deg(P_i)}\sum_{h=1}^n \varepsilon(x^h)i(x_h)P_i=P_i.$$
(Recall that $P_i$ is homogeneous 
of odd degree. This allows us to swap the exterior multiplication by $x^h$ over to 
the left: the second to last equality follows).\par
To prove the second statement, compute using Lemma \ref{conto}:  \begin{align*}(\cac\otimes1)u_i&=\frac{2}{\deg(P_i)}\sum_{h=1}^n\cac\partial i(x_h)P_i\otimes x^h=-\frac{2}{\deg(P_i)}\sum_{h=1}^n\cac  i(\cac x_h)P_i\otimes x^h \\&= \frac{1}{\deg(P_i)}\sum_{h=1}^n   i( x_h)P_i\otimes x^h=f_i.\end{align*}
The last identity $(\cac\otimes 1) f_i=0$ follows from $\cac^2=0$.\end{proof}
Since we have that  $\dim A= 2^rr$ (see \cite[Corollary 50]{K}), Theorem \ref{main}  will follow from  
\begin{proposition}\label{fun}
The elements $f_i,u_i $ are linearly independent over $\bigwedge(P_1,\ldots,P_{r-1}).$
\end{proposition}
In fact,  Proposition \ref{fun}  can be reduced to
\begin{lemma}\label{Ma}
The elements $f_i  $ are linearly independent over  $\bigwedge(P_1,\ldots,P_{r-1}).$
\end{lemma}
Taking this lemma for granted (the proof is at the end of this section), we can prove Proposition \ref{fun}.
\begin{proof}[Proof of Proposition \ref{fun}]
 Suppose that we have a relation $\sum\limits_{i=1}^r\lambda_iu_i+\sum\limits_{j=1}^r\mu_jf_j=0$. Then  apply $\cac\otimes 1$ and get $\sum\limits_{i=1}^r\lambda_if _i=0$,  so if we assume that the $f_i$ are linearly independent, we get $\lambda_i=0$ for all $i$  and in turn that  also all the $\mu_j$ are $0$.
\end{proof}  Notice that the only property we have used of the $u_i$ is the fact that $(\cac\otimes 1)u_i=f_i$.

\subsection{The main form\label{baf}} The Killing form on $\g$ induces an invariant  graded symmetric bilinear form on $\bigwedge  \mathfrak g\otimes\mathfrak g $ with values in  $\bigwedge  \mathfrak g $,
given by
\begin{equation}\label{formae}e(a\otimes x,b\otimes y)=(x,y)a\wedge b\end{equation}
for  $x,y\in \mathfrak g$, $a,b\in \bigwedge  \mathfrak g $.  In particular, by invariance,  this form induces a form on $A$ with values in $\Gamma=\bigwedge(P_1,\ldots,P_{r}).$
\subsubsection{Some pairings} 
Recall the following   properties of the exponents, (cf. \cite{Bou}): $m_1=1$ while $ m_r=h-1$, where $h$  is the Coxeter number. Also, the $m_i$ come naturally into pairs adding to  $h$.  We say that two indices $i,j=1,\ldots,r$ are {\it complementary} if the corresponding exponents sum to the Coxeter number. Since  we have ordered the exponents increasingly, the indices  $i,r-i+1$ are complementary. The exponents are all distinct except in the case of $D_{2n}$ where the exponent $2n-1$ has multiplicity two.
\vskip10pt
Consider  the matrix of scalar products $e(f_i,u_j)$; we will prove  that, for complementary indices $i,r-i+1$  we have
\begin{equation}\label{t}
 e(f_i,u_{r-i+1}) = c_i P_r,\quad 0\neq c_i\in\mathbb Q
\end{equation} 
if $\g$ is not of type $D_{2n}$. In this case we have  $m_n=m_{n+1}=2n-1$, and we can choose the primitive invariants   in such a way that \eqref{t}
holds and furthermore:
\begin{equation}\label{t1}
 e(f_n, u_{n})=e(f_{n+1},u_{n+1})=0.
 \end{equation} 
Formulas \eqref{t}, \eqref{t1} will be proved in \S \,\ref{lacu}. Assuming these formulas we can finish the proof of Lemma \ref{Ma}, hence that of   Theorem \ref{main}. 
\begin{proof}[Proof of Lemma \ref{Ma}] 
Remark that, if there is a non trivial relation $\sum\limits_{j=1}^r\mu_jf_j=0$,  we may assume that it  is homogeneous, that is all terms have the same degree.  Moreover, given an index $j$,    multiplying by a suitable element of  $\bigwedge(P_1,\ldots,P_{r-1}) $ we can reduce ourselves to the case in which $\mu_j= P_1\wedge P_2\wedge \ldots \wedge P_{r-1} .$

Notice now  that the   coefficient $\mu_h$ of the terms $\mu_hf_h$ for which $m_h<m_j$  has  degree higher than the maximum allowed degree, hence it is zero. Thus, if we choose for $j$ the maximum for which $\mu_j\neq 0$  and  if the exponent $m_j$ has multiplicity $1$, we are reduced to prove that  
\begin{equation}\label{f}P_1\wedge P_2\wedge \ldots \wedge P_{r-1} f_j\neq 0.\end{equation} The previous condition is satisfied for all exponents unless    
 we are in the case of $D_{2n}$. In this case the exponent $2n-1$ appears twice. The corresponding elements $f_n ,f_{n+1}$ have the same degree $4n-2$, so that we have to show that $P_1\wedge P_2\wedge \ldots \wedge P_{r-1} f_n$ and $P_1\wedge P_2\wedge \ldots \wedge P_{r-1} f_{n+1}$ are linearly independent.

Unless we are in the case $D_{2n}$ for the exponent of multiplicity two,  we are reduced to prove \eqref{f}. 
By \eqref{t} we have
$$ e(P_1\wedge P_2\wedge \ldots \wedge P_{r-1} f_j,u_{r-j+1})= c_j\ P_1\wedge P_2\wedge \ldots \wedge P_{r-1} \wedge P_r\neq 0.$$
In the remaining case, linear independence of     $P_1\wedge P_2\wedge \ldots \wedge P_{r-1} f_n$ and $P_1\wedge P_2\wedge \ldots \wedge P_{r-1} f_{n+1}$   follows in the same way using also formula \eqref{t1}.\end{proof}\subsection{Proof of Formulas \eqref{t} and \eqref{t1} }

The proof of  formulas \eqref{t} and \eqref{t1} will be    based on a computation in the symmetric algebra, which in turn, by the Chevalley restriction theorem, is deduced from the computations performed in \S \ref{iW}. 
\subsection{Invariants of the Weyl group\label{iW}} The study of invariants both in the symmetric and the exterior algebra of $\g^*$ depends on special properties of the reflection representation of the Weyl group, in particular of the action of a Coxeter element. We refer to \cite{Bou} for the details. Let $c$  be a Coxeter element of $W$ (i.e., a product of all simple reflections). Recall that  its order $h$ is the Coxeter number.  Set $\zeta=e^{\frac{2\pi i}{h}}$. The eigenvalues of $c$ in its reflection representation $\mathfrak h^*$ are $h$-th roots of 1, $\zeta^{m_i},\ 1\leq m_i<h$ and the $m_i$ are by definition {the exponents. Since the reflection representation is real, together with any eigenvalue $\eta$  also appears its conjugate $\bar\eta=\eta^{-1}$.  
 
Fix a basis $\{X_1,\ldots,X_r\}$ of $\mathfrak h^*$ made of eigenvectors relative to these eigenvalues. In particular, $X_1$ is relative to the eigenvalue  $\zeta$. 
We pair  conjugate eigenvalues,  and the Killing form on the space spanned by two conjugate eigenvectors $X_i,\ X_{r-i+1}$  is (after normalization) $\begin{pmatrix} 0&1\\1&0
\end{pmatrix}$.  In the case of the eigenvalue  $-1$, which is conjugate of itself, we either have a single eigenvector, which we may assume to be of norm 1, or  two eigenvectors    in type $D_{2n}$.  Then in  \cite[6.2, Proposition 2] {Bou} it is shown that to each exponent $m_i$ we may associate a $W$--invariant polynomial on $\mathfrak h$  such  that its leading term    in the variable $X_1$  is $X_1^{m_i }X_{r-i+1}$. 

\subsection{Symmetric algebra}
  Let  $S=S(\g^*)$  be the algebra of polynomial functions  on $\mathfrak g$, and $R=S^{\mathfrak g}$ be  the ring of invariants.  We have that $R$ is a polynomial algebra in generators $\psi_1,\ldots, \psi_r$ of degrees $m_i+1$.  We  observe that these invariants are so that when we restrict $\psi_i$ to $\mathfrak h$ the leading term in $X_1$ is a nonzero multiple $X_1^{m_i }X_{r-i+1}$,
 as we have seen in \S \ref{iW}.\par
 Let  $d$ be the usual differential of functions so that, when $f\in R$,  we have that $df\in(S\otimes \mathfrak g^*)^{\mathfrak g}$. 
\begin{definition}\label{bezout}
The    Bezoutiante matrix, 
 is the matrix (with entries in $R^+$) of the scalar products $(d\psi_i,d\psi_j)$.
 \end{definition}
 \par
We first point out  that the construction of the Bezoutiante is compatible with  Chevalley restriction.
\begin{lemma}\label{redbez}
Given two invariants $p,q\in R$  let $\bar p,\bar q$ be their restriction to $\mathfrak h$. Then
$$\overline{(dp,dq)}=(d\bar p,d\bar q).$$
\end{lemma}\begin{proof}
Use the decomposition $\mathfrak g=\mathfrak h\oplus (\mathfrak u^+\oplus \mathfrak u^-)$, which is orthogonal with respect to the Killing form. Write a polynomial on $\mathfrak g$ in the corresponding coordinates associated to root vectors, if  $F$ is such a polynomial we divide it as $F=\bar F+f$  where $\bar F$ collects only the monomials in the coordinates of $\mathfrak h$ while every monomial in $f$ contains at least one coordinate $x_\alpha$ associated to a root. Now if  $F$ is invariant in particular each monomial must have weight 0 for $\mathfrak h$, hence if some variables $x_\alpha$ appear  for a positive root, there must be other variables for negative roots. This means that $f$ vanishes of order at least 2 on $\mathfrak h$, hence also its differential vanishes and we have the claim by the orthogonality relation.
\end{proof}
 \begin{proposition}\label{lap} The entries of the matrix  $( d\psi_i,d\psi_j )$ below the  antidiagonal lie in $(R^+)^2$. On the antidiagonal they are    non zero multiples of $\psi_r$ modulo squares. \end{proposition}
\begin{proof} The first statement follows by degree considerations.
As for the second, by Lemma \ref{redbez} we can compute the Bezoutiante matrix on the $W$-invariant functions $\bar\psi_i$ on $\mathfrak h$.  

 The form in the basis $X_1,\ldots,X_r$ introduced above, equals $\sum_{(i,	\sigma(i))}X_iX_{\sigma(i)}$ (where $\sigma$ is the involution pairing complementary indices,  the sum is over the orbits of $\sigma$ which are made by two  or one element).         It follows that  the invariant  $(d\bar\psi_i,d\bar\psi_j)$  has leading term in $X_1$   proportional to   $X_1^h$  if $i,j$  are complementary  indices, arising from $$\pd{X_1^{m_i }X_{j}}{X_{j}}\pd{X_1^{ m_j }X_{i}}{X_{i} } =X_1^h.$$ 
 Hence, if $i+j=r+1$,   by the previous relation we have $(d\bar\psi_i,d\bar\psi_j)=  c_i\bar\psi_r\mod (R^+)^2$  for some  nonzero $c_i$.  
 \end{proof}
\subsection{Transgression}  
Consider the map $\cac:\mathfrak g^*\to \bigwedge^2 \mathfrak g^*$. Since the elements in its image commute with respect to exterior multiplication, $\delta$ admits a unique extension to   an algebra homomorphism$$s:S(\g^*)\to \bigwedge {}^{even} \mathfrak g^*.$$  Observe that $s$ takes as values coboundaries, in particular $\cac s=0$. Using   the multiplication map $m$ defined by \eqref{multiplication}
we set, for   $a\in S(\g^*)$: 
\begin{equation}
\label{tras}\pip 
(a):=m((s\otimes 1)da).
\end{equation}

It follows from the explicit description  of the transgression map, which is obtained in \cite{Chev} (and in formula (231) of  \cite[Theorem 64]{K}), that on homogeneous invariants of positive degree $d$ our map $t$ coincides up to the constant  $ {(d!)^2}/{(2d+1)!}$ with the classical transgression map. See also \cite[Remark 6.11]{Mein} Êfor a discussion on terminology.
We prefer to eliminate this constant since it is cumbersome in the formulas. The key result on the transgression map \cite{Car} can be reformulated as follows.

\begin{theorem}\label{proptransgr} The map $\pip$ vanishes on $(R^+)^2$ and induces an isomorphism of $ R^+ /(R^+)^2$ with the space of primitive elements of $\Gamma$.

Thus, given generators $\psi_i$ of degree $m_i+1$  for   $R$,  one has that the elements $P_i:=\pip( \psi_i)$  are  primitive generators of $\Gamma$, of degrees $2m_i+1$.
\end{theorem}  

In particular, choose  a set $\{\psi_1,\ldots,\psi_r\}$ of homogeneous generators of the ring $R$ of symmetric invariants. We obtain a set of primitive generators $P_i:=\pip( \psi_i)$ of $\Gamma$ and the corresponding elements $u_i,f_i$ defined by formula \eqref{deffu}. To proceed further, we need some general formulas.  
\begin{lemma} For all  homogeneous elements $a,b\in S(\g^*)$ we have
 \begin{align}\label{fov1}
\pip 
(ab) &=\pip 
(a)\wedge s(b)+s(a)\wedge \pip 
(b),\\\label{fov2}
\cac \pip 
(a)&=\deg(a)     s(a).
\end{align}
\end{lemma}
\begin{proof}
\begin{align*} \pip 
(ab)&=m((s\otimes 1)d(ab) )=m((s\otimes 1)(bd(a )+ad(b) ))\\&=m(s(b)\wedge (s\otimes 1) d(a ))+m(s(a)\wedge  (s\otimes 1) (d(b) )\\&=s(b)\wedge m( (s\otimes 1) d(a ))+s(a)\wedge m( (s\otimes 1) (d(b) )= \pip 
(a)\wedge s(b)+s(a)\wedge \pip 
(b).\end{align*}
This proves \eqref{fov1}. To prove \eqref{fov2},   let $da= \sum\limits_{h=1}^n  i_S(x_h )a\otimes x^h$ (where $i_S(x_h)$ is the directional derivative w.r.t. $x_h$),  so that $\deg(a)a=\sum\limits_{h=1}^n i_S(x_h ) x^h$ and $\pip 
(a)=\sum\limits_{h=1}^n s(i_S(x_h ))\wedge x^h$. Then 
\begin{align*} \cac \pip 
(a)&=\sum_{h=1}^n\cac s(i_S(x_h ))\wedge x^h+\sum_{h=1}^n s(i_S(x_h ))\wedge \cac x^h=\sum_{h=1}^ns(i_S(x_h ))\wedge \cac x^h\\&=\sum_{h=1}^n s(i_S(x_h ))\wedge s(x^h)=s(\sum_{h=1}^n i_S(x_h ) x^h)=\deg(a)s(a).\end{align*} 
\end{proof}\begin{definition}
Define for any homogeneous element $a\in S(\g^*)$:
\begin{equation}
\label{nuf}f(a):= (s\otimes 1)da,\quad u(a):= (t\otimes 1)da.
\end{equation}
\end{definition}
Note that, since  $da=\sum\limits_{h=1}^n i_S(x_h )a\otimes x^h$, we have $\deg(i_S(x_h )a)=\deg(a)-1$,  \eqref{fov2} implies
\begin{equation}\label{nnuf}
(\cac \otimes1)u(a)=(\deg(a)-1)f(a).
\end{equation}
\begin{lemma}
For every $i=1,\ldots ,r$, take $f_i$ as in formula \eqref{gliuf}. Then
 $\label{foto}f_i= f(\psi_i)$.\end{lemma} 
\begin{proof} 
Set for $a\in R$ homogeneous:$$f_a:=\frac{1}{  2\deg(a)+1}\sum_{h=1}^n i(x_h)t(a)\otimes x^h,$$ so that $f_i=f_{\psi_i}$ by \eqref{gliuf}.

By  \cite{K}, Theorem 73, one has that $i(x )t(a)=c s(i_S(x  )a)$ for any $x$ and $a$ invariant, where the constant $c$ depends on the normalizations we choose for $\pip$. Thus, since $(s\otimes 1)(da)=  \sum\limits_{h=1}^n s(i_S(x_h )a)\otimes  x^h,$  we have: $$f(a)=(s\otimes 1)(da)=  \sum_{h=1}^n s(i_S(x_h )a)\otimes  x^h= \sum_{h=1}^n c \  i (x_h  )t(a)\otimes  x^h,$$ so $f_i$ and $f(\psi_i)$ are proportional. To determine the proportionality constant $c$ it is enough to apply the multiplication map $m$.  $f_i$ has been normalized so that $m(f_i)=P_i$, while $m(f(\psi_i))=t(\psi_i)=P_i$ by definition.  Thus  $c=1/( 2\deg(a)+1) $ and everything follows.
\end{proof}

  \begin{proposition}\label{fine}  Given two elements  $a,b\in S(\g^*)$ we have 
\begin{equation}
\label{cc}e(f(a),f(b))= s((da,db)),\quad  e(u(a),f(b))+ e(f(a),u(b))=  \pip 
((da,db)).\end{equation} 
\end{proposition}
\begin{proof}
By construction $f 
(a)=(s\otimes 1)d(a),\quad f 
(b)=(s\otimes 1)d(b)$ hence 
$$ e(f(a),f(b))=s((da,db)).$$
 To prove the second relation,  write $da=\sum\limits_{j=1}^n \alpha_j\otimes z_j,\ db=\sum\limits_{j=1}^n \beta_j\otimes z_j$ where $\{z_j\}$ is an orthonormal basis. Then $(da,db)=\sum\limits_{j=1}^n  \alpha_j \beta_j$, hence by \eqref{fov1} and \eqref{nuf} we have
\begin{align*}
t((da,db))&=\sum\limits_{j=1}^n \pip 
 (\alpha_j) \wedge s(\beta_j)  + \sum\limits_{j=1}^n  s(\alpha_j) \wedge \pip 
(\beta_j)\\
&=e((\pip 
\otimes 1)(da),(s\otimes 1)(db))+e((s\otimes 1)(da),(\pip 
\otimes 1)(db))\\
&=  e(u(a),f(b))+ e(f(a),u(b)),\end{align*} 
which is the required relation.\end{proof}
\begin{lemma}\label{lacc1}
For $a,b\in\bigwedge \g^*\otimes \g^*$, we have  that 
 \begin{equation}
\label{lacc}\partial e(a,b)- e(\partial\otimes 1(a), b) +(-1)^{\deg(a)}e(a, \partial\otimes 1(b))
\end{equation}   is orthogonal (for the   Killing form)  to all invariant  elements, hence zero if it is an invariant.
\end{lemma}
\begin{proof}
Write,  in an orthonormal basis $\{z_i\}$, $a=\sum\limits_{i=1}^n a_i\otimes z_i, b=\sum\limits_{i=1}^n  b_i\otimes z_i,  e(a,b)=
 \sum\limits_{i=1}^n a_i\wedge b_i$. By formula (4.7) of \cite{Ko} we have that $\sum\limits_{i=1}^n ( \partial(a_i\wedge b_i)- \partial(a_i)\wedge b_i +(-1)^{\deg(a_i)}a_i\wedge \partial(b_i))$  is orthogonal (for the scalar Killing form)  to all invariant  elements. Since
 $\partial\otimes 1(a)= \sum\limits_{i=1}^n \partial a_i\otimes z_i,\,\partial\otimes 1(b)= \sum\limits_{i=1}^n \partial b_i\otimes z_i,$
the claim follows.
\end{proof}
We now set for $a\in S(\g^*)$ a  homogeneous invariant$$\bar u(a):=2(\partial\otimes 1)f(a)$$ 

 \begin{proposition}
If $a,b\in S(\g^*)$ are homogeneous invariants, then 
\begin{align}
\label{ccinv}e(f(a),f(b))&= 0,\\
\label{relu}e(\bar u(a),\bar u(b))&= 0,\\
\label{latra}e(\bar u(a),f(b))&=e(f(a),\bar u(b))= \frac{ \pip 
((da,db))}{\deg(a)+\deg(b)-2}.\end{align} 
\end{proposition}
\begin{proof}
From \eqref{cc}  and the definition of $s$  we have that $e(f(a),f(b))$  is the value of $s$ on an invariant.  Now remark that elements in the  image of $s$ are coboundaries, and an  invariant which is also a coboundary is   $0$. This proves \eqref{ccinv}.
\par In order to prove   \eqref{relu}, we 
 compute \eqref{lacc} for $\bar u(a),f(b)$. Using Lemma \ref{lacc1} we get
 \begin{align*}
0=&\partial e(\bar u(a),f(b))- e((\partial\otimes 1)(\bar u(a)), f(b)) +(-1)^{\deg(\bar u(a))}e(\bar u(a), (\partial\otimes 1)(f(b)))\\
=&(-1)^{\deg(\bar u(a))}\tfrac12e(\bar u(a), \bar u(b) ),
\end{align*}  since both $ \partial e(\bar u(a),f(b)),\ e((\partial\otimes 1)(\bar u(a)), f(b)) $ are clearly 0.

 We pass now to the last  identity. We have for any homogeneous element $w\in A$,
\begin{align*}0&=\delta (e(w,\bar u(b)))=e((\delta\otimes 1)(w),\bar u(b))+(-1)^{\deg(w)}e(w,(\delta\otimes 1)(\bar u(b)))\\&=e((\delta\otimes 1)(w),\bar u(b))+(-1)^{\deg(w)}e(w,f(b)).\end{align*}
It follows that $$e(w,f(b))=-(-1)^{\deg(w)}e((\delta\otimes 1)(w),\bar u(b)).$$
Since  $f(a)=1/(\deg(a)-1)\delta\otimes 1 u(a)=  \delta\otimes 1 \bar u(a),$  we deduce
$$e( u(a),f(b))=(\deg(a)-1)e(\bar u(a),f(b)).$$
We  have   $\cac   e(\bar u(a),\bar u(b))=0$.  Computing and using the fact that  $\bar u(a)$ has odd degree, we get 
$$0=\cac   e(\bar u(a),\bar u(b))= e((\cac \otimes1)(\bar u(a)),\bar u(b))-  e(\bar u(a),(\cac \otimes1)(\bar  u(b)))= e(f(a),\bar u(b))-  e(\bar u(a), f(b)),$$ 
that is $e(f(a),\bar u(b))=e(\bar u(a), f(b)).$ Formula   \eqref{latra} then follows immediately from \eqref{cc}.\end{proof}
\begin{corollary}
\begin{align}
\label{ccinvi}e(f_i,f_j)&= 0,\\\label{relui}e( u_i ,  u_j)&= 0,\\\label{latrai}e(  u_i,f_j)&=e(f_i ,\ u_j)= \frac{  \pip 
((d\psi_i,d\psi_j))}{\deg(\psi_i)+\deg(\psi_j)-2}.\end{align} 
\end{corollary}
\vskip5pt
\subsubsection{Examples in the classical cases}\label{ddd} \vskip10pt
For elements of $R$ we introduce the following notation:  $A\cong B$ if the
two invariants are congruent modulo $(R^+)^2$.   By abuse of notation we denote by the same symbol the invariant on $\g$ or its restriction to $\mathfrak h$.\vskip10pt
\noindent {\sl Type $A_n$.} Consider  as  generating invariants the {\em normalized} Newton polynomials
\begin{equation}
\label{nN}q_{k}:=\frac1k\sum_{i=1}^{n+1}x_i^k,\ k=2,\ldots, n+1
\end{equation}  
(so that, in our general notation, $\psi_k=q_{k+1}$, $k=1,\ldots,n$).  We have
\begin{equation}
\label{nNp}(dq_{k},dq_{g})=  \sum_{i=1}^{n+1}x_i^{k+g-2}=(k+g-2)q_{k+g-2}.
\end{equation}
We associate, via transgression,  to these symmetric invariants the primitive invariants  $P_i:=t( \psi_{i})=m((s\otimes 1)d \psi_{i})$.  We have the corresponding  $f_i, u_i$ and, plugging \eqref{nNp} into  \eqref{latra},
we obtain: 
\begin{equation}\label{ef}e(f_i,u_j)= \frac{t( (i+j) q_{i+j})}{i+j}=\begin{cases}
 P_{i+j-1}\quad &\text{if  $i+j\leq n+1$,}\\0\quad&\text{otherwise.}
\end{cases}.\end{equation}
{\sl Type $C_n$.} Here we can choose as  generators  the normalized Newton functions $\psi_i=q_{2i},\ i=1,\ldots,n$. The formulas are the same as before:
\begin{equation*}\label{efB}e(f_i,u_j)= t(q_{2i+2j-2})=\begin{cases}
 P_{i+j-1}\quad &\text{if  $2i+2j-2\leq 2n$,}\\0\quad&\text{otherwise.}\end{cases}.\end{equation*}
 which is indeed \eqref{ef}.
\vskip5pt\noindent{Type $D_{2n+1}$, .} Here we can choose as  generators  the normalized Newton functions $\psi_i=q_{2i},\ i=1,\ldots,n$,  the function $\psi_{n+1}=\mathcal P:=\sqrt{-2n} \prod_{i=1}^{2n+1}x_i$ and the functions $\psi_i=q_{2i-2}$ for $i=n+2,\ldots ,2n+1$. The formulas are the same as before for the Newton functions. As for those involving $\mathcal P$, we have 
\begin{align}
\label{pr0}&(d\mathcal P,dq_{2})=(2n+1)\mathcal P,\\
\label{pri}&(d\mathcal P,dq_{2k})=2(k-1)q_{2k-2}\cdot \mathcal P\in (R^+)^2,\ \forall k>1,\\
\label{sec}&(d\mathcal P,d\mathcal P)\cong 4n\, q_{4n}.
\end{align}
\begin{proof}[Proof of \eqref{pr0}, \eqref{pri}, \eqref{sec}]
Formulas \eqref{pr0}, \eqref{pri} are straightforward. As for \eqref{sec} we have\begin{equation}
\label{Dn}d\mathcal P=\sqrt{-2n}\sum_{i=1}^{2n+1}(\prod_{j\neq i}x_j)dx_i\implies (d\mathcal P,d\mathcal P)= -2n \sum_{i=1}^{2n+1}\prod_{j\neq i}x_j^2=-2n\, s_{2n}(x_1^2,\ldots,x_{2n+1}^2)
\end{equation}  where by $s_i$ we denote the $i^{th}$ elementary symmetric function. From the formulas expressing elementary symmetric function in terms of power sum we see that, for each $i\leq 2n+1$, we have $s_i(x_1 ,\ldots,x_{2n+1} )\cong(-1)^{i+1}q_{i}$  modulo squares,  hence 
\begin{equation}\label{prove}-2n\, s_{2n}(x_1^2,\ldots,x_{2n+1}^2)\cong 2n\, q_{2n}(x_1^2,\ldots,x_{2n+1}^2)=4n\,q_{4n}(x_1 ,\ldots,x_{2n+1}).
\end{equation}
Combining \eqref{Dn} and \eqref{prove} we get \eqref{sec}.\end{proof}
 
Now we can compute the scalar products $e(f_i,u_j)$; if $f_i,u_j$ are constructed from Newton functions we still have formulas \eqref{ef}. As for  $f_{n+1},u_{n+1}$, the elements associated to $\mathcal P$, we have, by formulas \eqref{pr0}, \eqref{pri} and \eqref{sec} 
\begin{equation}\label{perd}e(f_{n+1},u_j)=0,\ \forall j\neq n+1,\quad e(f_{n+1},u_{n+1})= \frac{t(4n\, q_{4n})}{4n}=P_{2n+1}.\end{equation}

\vskip5pt\noindent{Type $D_{2n}$, .} Here we can choose as  generators  the normalized Newton functions $\psi_i=q_{2i},\ i=1,\ldots,n-1$,  the functions $$\psi_{n}=q_{2n}+\sqrt{-1}\mathcal P,  \ \ \ \psi_{n+1}=q_{2n}-\sqrt{-1}\mathcal P,$$ with $\mathcal P:=\sqrt{2n-1} \prod_{i=1}^{2n}x_i$ and the functions $\psi_i=q_{2i-2}$ for $1=n+2,\ldots ,2n$. 

Notice that $\psi_i$ has degree $2i$ for $i=1,\ldots n-1$, has degree $2n$ for $i=n,n+1$ and has degree $2i-2$ for $i=n+2,\ldots ,2n$. The formulas are the same as before for the Newton functions. As for those involving $\mathcal P$, we have the formulas
\begin{align}
\label{pr00}&(d\mathcal P,dq_{2})=2n\mathcal P,\\
\label{prii}&(d\mathcal P,dq_{2k})=2(k-1)q_{2k-2}\cdot \mathcal P\in (R^+)^2,\ \forall k>1,\\
\label{secc}&(d\mathcal P,d\mathcal P)\cong (4n-2)\, q_{4n-2}.
\end{align}
whose proof is identical to that of formulas \eqref{pr0}, \eqref{pri}, \eqref{sec}.

Reasoning as before it immediate to verify that, up to  the case in which both $i$ and $j$ do not belong to $\{n,n+1\}$, we have
\begin{equation}\label{efff}e(f_i,u_j)= \begin{cases}
 P_{h}\quad &\text{if  $m_i+m_j-1=m_h,\,h\ne n,n+1$,}\\
 \frac{1}{2}(P_n+P_{n+1})&\text{if  $m_i+m_j-1=2n-1$,}\\
 0\quad&\text{otherwise.}\end{cases}\end{equation}
 On the other hand we have
\begin{equation}\label{ancoraD}
e(f_n,u_j)=\begin{cases}
 P_n\quad&\text{$j=1$}\\
  P_{n+j}\quad&\text{$2\leq j\leq n-1$}\\
  0\quad&\text{$j=n$}\\
  2P_{2n}\quad&\text{$j=n+1$}\\
  0\quad&\text{$j>n+1$}\\
  \end{cases}
  \quad
  e(f_{n+1},u_j)=\begin{cases}
P_{n+1}\quad&\text{$j=1$}\\
  P_{n+j}\quad&\text{$2\leq j\leq n-1$}\\
   2P_{2n}\quad&\text{$j=n$}\\
0 \quad&\text{$j>n$}
\end{cases}
\end{equation}

\subsubsection{Conclusion of the proof\label{lacu}}  
\begin{proof}[Proof of Theorem \ref{main}] Recall that we have reduced the proof of the theorem to the proof  of  formulas \eqref{t}, \eqref{t1}. The former follows from Proposition \ref{lap} and formula \eqref{latrai}. The latter, which is relative to $D_{2n}$, is obtained using also \eqref{perd}.\end{proof}

\begin{proof}[Proof of Theorem \ref{gendi}] 
Given an element $a\in R^+$ set $\tilde a$ equal to its image in $M:=R^+/(R^+)^2$. Take the usual set of homogeneous generating invariants $\psi_1,\ldots \psi_r\in R^+$ of degrees $d_i=m_i+1$. We have already seen that $\tilde \psi_i\circ \tilde \psi_j:=\widetilde{\psi_i\circ \psi_j}$ is well defined. Assume $\g$ not of type $D_{2n}$.  Then  our claim will follow if we show that $\tilde \psi_i\circ \tilde \psi_j\neq 0$ if and only if $m_i+m_j-1$ is an exponent.

The fact that $\tilde \psi_i\circ \tilde \psi_j=0$  if $m_i+m_j-1$ is not an exponent is clear since $M$ does not contain elements of degree $m_i+m_j=\deg (\tilde \psi_i\circ \tilde \psi_j)$.

When $m_i+m_j-1$ is an exponent, the explicit analysis of the classical cases performed above shows that $\tilde \psi_i\circ \tilde \psi_j$ is actually non zero, hence a multiple of the corresponding generator of degree $m_i+m_j$. The statement of Theorem \ref{gendi}  in type $D_{2n}$ follows by using formulas \eqref{pr00}-\eqref{secc}.

As for the exceptional cases, remark that the cases of $G_2$ and $F_4$ are  an immediate consequence  of Proposition \ref{lap}.

It remains to discuss the case of algebras of type $E$.
Let us now consider in each of  the types $E_6, E_7, E_8$, the equation $p(w,x,y)$ of the corresponding simple surface singularities (see \cite{Sl}). For convenience of the reader let us recall that they are 
\begin{align*}{} w^2+x^3+y^4\ \ \ \text {for\ type\ $E_6$},\\
 w^2+x(x^2+y^3)\ \ \ \text {for\ type\ $E_7$},\\
  w^2+x^3+y^5\ \ \ \text {for\ type\ $E_8$}.\end{align*}
Take the algebra $Q=\mathbb C[x,y]/(p_x,p_y)$. 

Let us  first point out a few simple facts about $Q$ which are readily verified.
$Q$ is a local algebra with maximal ideal $m=(x,y)$. By suitably choosing the degrees of the variables $w,x,y$, $p$ becomes in each case, a homogeneous polynomial of  degree equal to the Coxeter number $h$. It follows that  $Q$ is graded. An easy computation shows that $Q$ has Poincar\'e polynomial  given by
$\sum_{i=1}^rt^{m_i-1}$.

Consider now the operator $S$ on $Q$ which, for any homogeneous element 
$q\in Q$,  is given by$$Sq=\frac{\deg q+2}{h}q.$$
By Corollary 2 in \cite{Given} there is an isomorphism $\Psi:M\to Q$ such that
\begin{enumerate}\item If $\tilde a\in M$ is homogeneous then 
$$\deg(\Psi(\tilde a))=\deg(\tilde a)-2.$$
\item
For any $\tilde a,\tilde b\in M$,
$$\Psi(\tilde a\circ\tilde b)= S(\Psi(\tilde a)\Psi(\tilde b)).$$
\end{enumerate}
From this our claim follows by a straightforward case by case analysis.

\end{proof}

\begin{proof}[Proof of Propositions \ref{main1}, \ref{mainD}] Formulas \eqref{ccinvi} and \eqref{relui} show that the $u$'s and the $f$'s generate isotropic subspaces.

Let us now pass to the determination of $e(f_i,u_j)$.
Combining Theorem \ref{proptransgr}Ê and formula \eqref{ccinvi} we obtain that everything follows from Theorem \ref{gendi}.
\end{proof}

\subsection{Classical groups\label{clgr}}  

If $G$  is a  classical group one can take a different approach. Let $V$ be the defining  representation of $G$, so that  $\g=Lie(G)$ is a subalgebra of $End(V)$, which decomposes as $End(V)=\g\oplus \mathfrak p$  where in case $A_n$ the space $ \mathfrak p$ is the 1--dimensional trivial representation, while in the other cases it is the space of symmetric matrices for the corresponding involution, in all cases an irreducible representation. It is convenient
  to study the associative {\em invariant algebra}, i.e. the algebra of $G$--equivariant maps\begin{equation}
\label{an}A_G=(\bigwedge End(V)^*\otimes End(V))^G.
\end{equation} Then one can analyze inside $A_G$  the super Lie subalgebra $(\bigwedge \g^*\otimes \g)^G.$
This among other topics is  discussed in \cite{bps}.

In the natural basis $\{e_{ij}\}$ of matrices with  coordinates $x_{ij}$ consider the element $X\in A_{GL(N)}$ (cf. \eqref{an}), which is  the {\em   generic Grassmann matrix}\quad $X=\sum_{h,k}x_{hk}e_{hk} $. Its
power  $X^a=X^{\wedge a}$  equals the standard polynomial $S_a$  computed in $End(V)$, hence in this language the Amitsur--Levitzki Theorem, see \cite{P}, is the single identity $X^{2n}=0$. In \cite{bps} the authors prove, among other results, the following 
\begin{theorem}
\label{alg}  The algebra $A_{GL(n)}$ is generated by $X$  and the elements $tr(X^{2i-1}), \ i=1,\ldots, n.$  All these elements anticommute.

$A_{GL(n)}$  is a free module with basis $X^i,\ i=0,\ldots, 2n-1$ over the Grassmann algebra in the elements $tr(X^{2i-1}), \ i=1,\ldots, n-1$ and we have the two defining identities
\begin{equation}
\label{defi}X^{2n}=0,\quad tr(X^{2n-1})=-\sum_{i=1}^{n-1}X^{2i} \wedge  tr(X^{2(n-i)-1})+nX^{2n-1}.
\end{equation}\end{theorem}
As for the connection with $(\bigwedge \g^*\otimes \g^*)^\g\cong (\bigwedge \g^*\otimes \g )^\g$  where $\g=sl(n)$  we see immediately that
$$A_{GL(n)}=(\bigwedge \g^*\otimes \g)^\g\oplus\Gamma. $$ Write  $X=Y+\frac {tr(X)}{n}$ where $tr(Y)=0$ and see that  for $a>1$ we have $X^{2a}=Y^{2a}$.
Hence $X^{2a}=Y^{2a}$ is in $(\bigwedge \g^*\otimes \g )^\g$   while in the odd case we have that  $X^{2a+1}-\frac {X^{2a} tr(X)}{n}$ is in $(\bigwedge \g^*\otimes \g )^\g$.
These are up no normalizations the elements $f_i,u_i$ defined in \S \ref{201} in the present case.  Clearly the defining identities \eqref{defi} which describe  the associative algebra also allow us to compute all the relations for the Lie algebra case.

The other classical groups can be treated in a similar way: an investigation in this direction has been started in \cite{Dolce}  (cf. Section \ref{FR}).
\section{Proof of Theorem \ref{main2}}
\subsection{The relations} In order to complete the description of $A$ as module over  $\Gamma$ we only need  to express the elements $P_r\wedge u_i$ and $P_r\wedge f_i$ in our given basis. 

Consider the relation for $u_i$; we have
\begin{equation}
\label{conu}P_r\wedge u_i=\sum_{j=1}^r H_j\wedge u_j+\sum_{j=1}^r K_j\wedge f_j,
\end{equation}
where  $H_j,K_j\in \bigwedge(P_1,\ldots,P_{r-1}).$ Applying the differential $\cac\otimes1$ we get
\begin{equation}\label{p}P_r\wedge f_i=\sum_{j=1}^r H_j\wedge f_j.\end{equation}
Thus the relation for  $f_i$ involves only $f_j$'s. Also we have that the relation is homogeneous.

For each $j$, taking the scalar product with $u_{r-j+1}$, we have 
\begin{align*} P_r\wedge e(f_i,u_{r-j+1})&=  H_j\wedge e(f_j,u_{r-j+1})+\sum_{h\neq j}H_h\wedge e(f_h,u_{r-j+1})\\
&=  H_j\wedge c_j P_r +\sum_{h\neq j}H_h\wedge e(f_h,u_{r-j+	1}).\end{align*}
Since the terms $\sum_{h\neq j}H_h\wedge e(f_h,u_{r-j+1})  $  do not involve $P_r$, we must have
\begin{align}\notag&\sum_{h\neq j}H_h\wedge e(f_h,u_{r-j+1}) =0,\\
&\label{1}- e(f_i,u_{r-j+1})\wedge P_r=  H_j\wedge  c_j   P_r  \end{align}
If $i\neq j$,  we have that $e(f_i,u_{r-j+1})$ is not a multiple of $P_r$ and we deduce that  
$$ e(f_i,u_{r-j+1}) = -c_j  H_j.$$
If $i= j$,  we   deduce $H_j=0$, so finally \eqref{p}Ê becomes 
\begin{equation}
\label{lamol}P_r\wedge f_i+ \sum_{i\neq j} c_j^{-1}e(f_i,u_{r-j+1})\wedge f_j=0.
\end{equation}
Since $e(f_i,u_{r-i+1})=c_iP_i$, formula \eqref{lamol} is indeed formula \eqref{p1}, as required. We go back to formula \eqref{conu}, which we now write:
\begin{equation}\label{p3}P_r\wedge u_i= -\sum_{j=1}^r c_j^{-1}e(f_i,u_{r-j+1})\wedge u_j+\sum_{j=1}^r K_j\wedge f_j.\end{equation}
Take the scalar product of both sides of \eqref{p3} with $u_{r-j+1}$. We get
$$P_r\wedge e(u_i,u_{r-j+1})= -\sum_{j=1}^r c_j^{-1}e(f_i,u_{r-j+1})\wedge e(u_j,u_{r-j+1})+\sum_{j=1}^r K_j\wedge e(f_j,u_{r-j+1}).$$
Since $e(u_h,u_k)=0$, we deduce that 
$$K_j\wedge  e(f_j ,u_{r-j+1})+\sum_{i,\,i\neq j} K_i\wedge e( f_i ,u_{r-j+1})=0$$
which in turn, by \eqref{t}, becomes
\begin{equation*}c_jK_j\wedge   P_r +\sum_{i,\,i\neq j} K_i\wedge e( f_i ,u_{r-j+1})=0.
\end{equation*}
We claim that all $K_j$ are zero. In fact,    the only product containing $P_r$ is $c_jK_j\wedge   P_r $. On the other hand, it is clear that  each element of $\Gamma$ can be written in a unique way in the form $a+b\wedge P_r$ with $a,b\in\bigwedge(P_1,\ldots,P_{r-1})$. We deduce that $K_j=0$ for each $j$, and 
   the proof of Theorem \ref{main2} is completed.
\vskip10pt
\section{Final Remarks}\label{FR}
In view of Theorem \ref{main}, it  is natural to ask  whether,  for some other  irreducible representation $L$ of $\g$, the space of covariants of type $L$
$$A_L:= \hom_\g(L ,\bigwedge \mathfrak g^*)$$
is free over $\bigwedge(p_1,\ldots,p_{r-1})$  of predictable rank. In \cite{DCPPM} it is  shown that the techniques and the outline of proof 
of Theorem \ref{main} can be enhanced to prove that $A_L$ is free over $\bigwedge(p_1,\ldots,p_{r-1})$ of rank $2\dim L_0$ ($L_0$ being the $0$-weight space of $L$), in the following two cases:
\begin{itemize}
\item $L$ is the little adjoint representation of $\g$ (i.e., the irreducible representation with highest weight the highest short root);
\item $\g$ is of type $A_{n-1}$ and $L=S^n(V)$ is the $n$-th symmetric power of the
defining representation $V$ or $L=S^n(V)^*$.\end{itemize}
\par
Another related question is the study of covariants of (indecomposable) infinitesimal symmetric spaces $\g=\mathfrak k\oplus\mathfrak p$, more precisely of the spaces
$(\g\otimes\bigwedge\mathfrak p^*)^{\mathfrak k}$, under the assumption that $(\bigwedge\mathfrak p^*)^{\mathfrak k}$ is an exterior algebra (say on $r$ generators).  This can be successfully pursued in the classical cases using the associative superalgebra structure of \eqref{an}Ê as in Subsection \ref{clgr}. The final outcome, proved by  Dolce in \cite{Dolce} when $\g$ is simple, is that  $(\g\otimes\bigwedge\mathfrak p^*)^{\mathfrak k}$
is free of rank $4r$ over the exterior algebra generated by the $r-1$ invariants of degree less then the maximal one. Notice that our Theorem \ref{main}  fits into this picture by considering as $\g$ two copies of a simple Lie algebra switched by the flip involution.
\vskip10pt
\section{Appendix} Formulas for invariants of exceptional groups  can be found in several places, except case $E_8$ for which there are computer programs but  no clear references.  A possible method of computation that we partly followed  appears in a paper of Lee \cite{L}. We assume we are in a case in which the exponents are all distinct.
Let $\mathfrak h$ and $W$  be as usual.  For $a\in\mathfrak h^*$ and $m\in\mathbb N$ define $$p_m(a):=\sum_{w\in W} w(a^m)=\sum_{w\in W} w(a)^m.$$
\begin{lemma} Given an exponent  $m_i$ 
the set of $a$ for which $p_{m_i+1}(a)$  is a generator  is a non--empty open Zariski set.
\end{lemma}
\begin{proof}
The elements $a^m$ generate linearly the symmetric power   $S^m(\mathfrak h^*)$   and the operator  $x\mapsto 1/|W|\sum_{w\in W} w(x)$ is the projection on the  invariants  therefore it is not possible that for all $a\in\mathfrak h^*$ the element  $p_m(a) =\sum_{w\in W} w(a^m)$ is not a generator.\end{proof} 
Let now $m_i$  be an exponent  and $I_{m_1+1},\ \bar I_{m_1+1}$ be  the space of invariants of degree $m_i+1$ and the subspace of decomposable ones, respectively. Thus   $\dim(I_{m_i+1}/\ \bar I_{m_i+1})=1.$
Let $\pi: I_{m_1+1} \to  I_{m_1+1}/\ \bar I_{m_1+1}$ be the projection.
We thus have a map 
$ a\mapsto  \pi(p_{m_i+1}(a))$  which,    trivializing the quotient, we think of  as an invariant  $Q_i(a)$.
\begin{theorem}
An element  $a\in \mathfrak h^*$ is such that the set  $\{p_{m_i+1}(a)\mid 1\leq i\leq r\}$ is  a system of  generators  if and only if   $Q_i(a)\neq 0$ for every $i$.
\end{theorem}
\begin{proof}
The condition is tautological but it has an interesting    interpretation.
The fact that these elements are generators is equivalent to ask that the determinant of the associated Jacobian matrix  $J(a)$ is non--zero.
This determinant, as function of $a$  is non--zero by the previous  lemma.  By degree and  symmetry considerations we deduce that it is a  multiple, dependent on $a$,   of the Weyl    denominator $\prod_{\alpha\in\Delta^+}\alpha$. This   multiple, as a function of $a$, is  $\prod\limits_{i=1}^r Q_i(a)$ (up to a constant). Hence  $$J(a)=\prod_{i=1}^r Q_i(a)\prod_{\alpha\in\Delta^+}\alpha.$$ 
\end{proof}
In practice the previous Theorem is hard to use, so we will proceed in a different way.  We compute only the first three  invariants, choosing $a$  in the simplest possible way and using  the previous formula. Then compute the remaining ones by applying the construction $u\circ v$ which, by Theorem \ref{gendi}, will provide all the remaining generators (recall that the algebra $Q$ is generated by the two homogeneous elements $x,y$ in all cases). 

In order to check  that the three  invariants we have found are indeed generators, it is enough to show that they are not decomposable. This can be verified by choosing a suitable monomial which never appears in the decomposable invariants of that degree and showing that the invariant we found has non--zero coefficient for that monomial. This is quite easy and we verified it for our choices.
\subsection{Exceptional groups}  Recall the exponents of exceptional groups.      
 
{\centering
       \begin{tabular}{@{} lcr @{}} 
          \multicolumn{2}{c}{} \\
        Type    & Exponents &\qquad h\\\\
        $E_6$     & 1, 4, 5, 7, 8, 11 & \qquad 12 \\\\
                    $E_7$      & 1, 5, 7, 9, 11, 13, 17  & \qquad 18 \\\\
       $E_8$       & 1, 7, 11, 13, 17, 19, 23,  29   & \qquad 30 \\\\
      $F_4$ & 1, 5, 7 ,11  &  \qquad 12 \\\\
      $G_2$ &1, 5 &\qquad 6\\
       \end{tabular}
        \vskip15pt
       \centerline{\small Table 1: exponents for the exceptional types.}
}
       \vskip15pt

  In $G_2,F_4$  there are no particular normalizations involved. For $F_4$ it is enough to define the invariant of degree 12 from the scalar product of those of degree $6,8$.

    In this section we usually will write $A_i$ for an invariant of the root system of degree $i$ and implicitly for the corresponding invariant in $S(\g^*)$.   The invariant of degree 2 will always be taken to correspond to the Killing form.  Then, by the definition of  the composition $a\circ b=(da,db)$ between invariants, we have that  
    $$ A_2\circ A_i= 2i A_i, \forall i.$$  When we pass to the corresponding  elements $P_i=t(A_i)$,   the element $t(A_2)$ is denoted by $P_1$ and has degree 3; formula \eqref{latra}  implies $e(u_1, f_i)=\frac{ t(A_2\circ A_i)}i=2P_i$.
    If we wish to have value 1 for the constants $c_{1,i}$ it is enough to let  $A_2$  correspond to half the  Killing form.\bigskip
    
    \subsection{A sketch of the computations}  In cases $E_6,E_7$ we are computing, on a space of dimension $n=6,7$  and coordinates $x_1,\ldots,x_n$, the scalar product of invariants written as polynomials in the power sums $p_i=\sum\limits_{j=1}^nx_j^i$.  Therefore the scalar product of two such invariants $a,b$  is  $\sum_{i,j}\pd{a}{p_i}\pd{b}{p_j}(dp_i,dp_j)$, so it is necessary to compute a priori $(dp_i,dp_j)=\sum_{h,k}\pd{p_i}{x_h}\pd{p_j}{x_k}(dx_h,dx_k)$.  In each case the quadratic function expressing the scalar product  is of the form  $a p_2+bp_1^2$. This means that the matrix  of the scalar products $(\pd{ }{x_i},\pd{ }{x_j})$ is $a1_n+bM_n$ where $M_n$ is the $n\times n$ matrix with all entries equal to 1, hence $M^2_n=nM_n$.  The matrix  of the scalar products $(dx_i,dx_j)$ is 
\begin{align*}(a1_n+bM_n)^{-1}&=a^{-1}(1_n-\frac{b}{a}(1+\frac{bn}{a})^{-1}  M_n) \\
(a1_n+bM_n)^{-1}&=a^{-1}(1_n-b(a+ bn )^{-1}  M_n).\end {align*}  In the two cases we have
    $$\tfrac23(1-\tfrac19M_6),\quad  \tfrac12(1-\tfrac19M_7),$$
  respectively,   so that $$\quad p_i\circ p_j=\begin{cases}
\frac{2ij}3(p_{i+j-2}-\frac19p_{i-1}p_{j-1})\quad &\text{ in type $E_6$,}\\\\
\frac{ij}2(p_{i+j-2}-\frac19p_{i-1}p_{j-1})\quad &\text{ in type $E_7$.}
\end{cases}$$

In case $E_8$ we have that the scalar product is $p_2=\sum\limits_{i=1}^8x_i^2$. The invariants are polynomials in the power sums $p_{2i},\ i=1,\ldots,7$  and the {\em Pfaffian:} $\mathcal P=\prod_{i=1}^8x_i$, hence  (cf. \eqref{Dn}): 
\begin{equation*}p_i\circ p_j =
 i j\,  p_{i+j-2},\quad p_i\circ\mathcal P= i\,p_{i-1}\mathcal P,\end{equation*}
\begin{align*}\mathcal P\circ \mathcal P&=s_7(x_1^2,\ldots,x_8^2)=\tfrac{1}{5040}p_2^7 - \tfrac1{240} p_2^5 p_{4} + \tfrac1{48} p_2^3 p_{4}^2 - \tfrac1{48} p_2 p_{4}^3 + 
 \tfrac1{72} p_2^4 p_{6} - \tfrac1{12} p_2^2 p_{4} p_{6} \\&+ \tfrac1{24} p_{4}^2 p_{6} + 
\tfrac1{18} p_2 p_{6}^2 - \tfrac1{24} p_2^3 p_{8} + \tfrac18 p_2 p_{4} p_{8} - \tfrac1{12} p_{6} p_{8} + 
 \tfrac1{10} p_2^2 p_{10} - \tfrac1{10} p_{4} p_{10} - \tfrac16 p_2 p_{12} + \tfrac17p_{14}.\end{align*}
 With these formulas and the expressions of the basic invariants given below all computations can be easily reproduced.   \bigskip

{\it Formulas for $E_6$.} In this case the Weyl group contains the symmetric group $S_6$ as Weyl group of a root subsystem of type $A_5$, and the restriction of the reflection representation to  $S_6$ can be identified to the 6--dimensional permutation representation. Therefore we can express the  $E_6$ invariants as polynomials in symmetric functions; we will use the power sums $p_i$.   In \cite{KM} the invariants are expressed through elementary symmetric functions. Changing variables we obtain  for the invariants of degrees 2,5,6 the following polynomials
\begin{align*}
A_2= &\tfrac{p_1^2}2+\tfrac{3p_2}2\\
A_5=&\tfrac{11}{20} p_1^5  - 6 p_1^3 p_2 + \tfrac{27}{4} 
    p_1 p_2^2 + \tfrac{27}{2} p_1^2 p_3 - \tfrac{27}2 p_2 p_3 - \tfrac{27}2 p_1 p_4 + \tfrac{81}5p_5 \\
    A_6=&\tfrac{25}8 p_1^6  - \tfrac{99}8 p_1^4 p_2 - \tfrac{297}8 p_1^2 p_2^2 + 
\tfrac{243}8 p_2^3 + 270 p_1 p_2 p_3 - 135 p_3^2 + \tfrac{135}{2} p_1^2 p_4 - \tfrac{405}{2} p_2 p_4\\ &- 324 p_1 p_5 + 324 p_6 
\end{align*}   
Recall that  write  $A\cong B$  when two invariants are congruent modulo $(R^+)^2$. 
Computing the scalar products of these basic invariants we define  $$ A_8:=A_5\circ A_5, \quad   A_9:=A_5\circ A_6, \quad  A_{12}:=A_5\circ A_9$$
and we have
$$8A_5\circ A_9\cong 9A_6\circ A_8$$
The corresponding  constants $d_{i,j}$ are given in in the following table.

\begin{center}\vskip5pt
 \begin{tabular}{c | c c c c}
&5&6&8&9\\
 \hline\\
5&1&1&0&1\\\\
6&1&0&$\frac89$&0\\\\
8&0&$\frac89$&0&0\\\\
9&1&0&0&0 
\end{tabular}
\vskip15pt
 \end{center}
 \centerline{\small Table 2: coefficients $d_{ij}$ for $E_6$.}
\medskip
\
{\it Formulas for $E_7$.} In this case the Weyl group contains the symmetric group $S_7$ as Weyl group of a  root subsystem of type $A_6$, and the restriction of the reflection representation to  $S_7$ can be identified to the 7--dimensional permutation representation. Therefore we can express the  $E_6$ invariants as polynomials in symmetric functions; we will use the power sums $p_i$.   In \cite{KM} the invariants are expressed through elementary symmetric functions. We perform the change of variables obtaining for the invariants of degrees 2,6,8. One can normalize the two invariants of degrees 6,8 so that:
\begin{align*}
A_2=&p_1^2 + 2 p_2\\
A_6=&\tfrac{10}3(p_1^6 - 12 p_1^4 p_2 + 36 p_1^2 p_2^2 - 6 p_2^3 + 40 p_1^3 p_3 - 
 120 p_1 p_2 p_3 + 40 p_3^2 - 60 p_1^2 p_4 + 60 p_2 p_4 +\\& 
 144 p_1 p_5 - 96 p_6) \\A_8=&\tfrac{10}7(p_1^8 + 224 p_1^6 p_2 - 1680 p_1^4 p_2^2 + 840 p_1^2 p_2^3 + 
 420 p_2^4 - 1568 p_1^5 p_3 + 12320 p_1^3 p_2 p_3 - \\&
 12320 p_1^2 p_3^2 - 4480 p_2 p_3^2 + 5040 p_1^4 p_4 - 
 18480 p_1^2 p_2 p_4 - 3360 p_2^2 p_4 + 20160 p_1 p_3 p_4 - \\&
 1680 p_4^2 - 18816 p_1^3 p_5 + 12096 p_1 p_2 p_5  + 2688 p_3 p_5 + 
 33600 p_1^2 p_6 + 6720 p_2 p_6 - 34560 p_1 p_7)\end{align*}
 We have normalized the invariants $ A_2, A_6,A_8$ so that,  computing the scalar products of these basic invariants, we get $$ A_{10}:=A_6\circ A_6, A_{12}:=A_6\circ A_8, A_{14}:=A_6\circ A_{10}\cong  
A_8\circ A_8,\ A_{18}:=A_6\circ A_{14}$$ and we verify that 
 $$  6A_{18}\cong 7 A_8\circ A_{12},\   5 A_{18}\cong 7 A_{10}\circ A_{10}$$
The constants $d_{i,j}$  are given in in the following table.

\begin{center}\vskip5pt
 \begin{tabular}{c | c c c c c}
& 6& 8&10&12&14\\
\hline\\
6&1&1&1&0&1\\\\
8&1&1& 0&$\frac{6}{7}$&0\\\\
10&1&0&$\frac{5}{7}$&0&0\\\\
12&0&$\frac{6}{7}$&0&0&0\\\\
14&1&0&0&0&0\\\\
\end{tabular}
\vskip15pt
 \end{center}
 \centerline{\small Table 3: coefficients $d_{ij}$ for $E_7$.}

 \bigskip

 Formulas for $E_8:$
The extended diagram of $E_8$ contains a subdiagram of type $D_8$,  so we can write the invariants of $E_8$ as polynomials in the invariants of $D_8$, which are generated by $p_{2i},\ i=1,\ldots,7$  and by the Pfaffian $\mathcal P$.

We have computed the invariants $A_8,A_{12}$ of degree 8 and 12 by the method of Lee \cite{L} and the others by taking the scalar products  $(da,db)$ starting from these two invariants.   The resulting invariants are generators by a simple inspection of their leading terms.\medskip

\begin{align*}
A_2=&p_2\\
A_8= &-10080 \mathcal P - 105 p_2^2 p_4 + 105 p_4^2 + 168 p_2 p_6 - 180 p_8\\
A_{12}= &-103950 \mathcal P p_2^2 +  \tfrac{10395}{64} p_2^6  + 41580 P p_4 - 
 \tfrac{51975}{32} p_2^4 p_4 + \tfrac{51975}{16} p_2^2 p_4^2 - \tfrac{5775 }{8}p_4^3  \\&+ 
 3465 p_2^3 p_6 - 6930 p_2 p_4 p_6 + 2772 p_6^2 - 
 \tfrac{25245}{4} p_2^2 p_8 + \tfrac{10395}{2} p_4 p_8 + 8316 p_2 p_{10} - 
 7560 p_{12}\end{align*}\medskip

We have the following  relations:
$$  A_8\circ (A_8\circ A_{12}) \cong\tfrac97  (A_8\circ A_8)\circ A_{12}    $$$$ (A_8\circ A_8)\circ(A_8\circ A_{12}) \cong\tfrac34  A_8\circ ((A_8\circ A_8)\circ A_{12}) ,\   A_{12}\circ (A_8\circ (A_8\circ A_8) )\cong  \tfrac56  A_8\circ ((A_8\circ A_8)\circ A_{12}) , $$
that is the invariant of order 24 computed in two different ways gives two different values (modulo squares) as well as the invariant of order  30 computed in  
three  different ways.

Set 
\begin{align*}A_{14}&:= A_8\circ A_{8} ,\\A_{18}&:= A_8\circ A_{12} ,\\
A_{20}&:=  A_8\circ (A_8\circ A_8) = A_8\circ A_{14} ,\\ A_{24}&:= (A_8\circ A_8)\circ A_{12} \cong  A_{12}\circ A_{14} ,\\A_{30}&:=  A_8\circ ((A_8\circ A_8)\circ A_{12}) \end{align*}
so that 
$$ A_{14}\circ A_{18} \cong  \tfrac34A_{30},\quad  A_{20}\circ  A_{12} \cong \tfrac56 A_{30}.$$
We deduce the matrix of the constants $d_{i,j}$, which is displayed  in the following table.\vskip5pt

\begin{center}\vskip6pt
 \begin{tabular}{c | c c c c c c c }
&  8&12&14&18&20&24& \\
\hline\\
 8&1&1&1&$\frac 97$ & 0&1& \\\\
12&1&0&1&0&$\frac56$&0& \\\\
14&1&1& 0&$\frac34$&0&0& \\\\
18&$\frac97$&0&$\frac34$&0&0&0& \\\\
20& 0&$\frac56$&0&0&0&0& \\\\
24&1& 0&0&0&0&0& 
\end{tabular}
\vskip15pt
 \end{center}
  \centerline{\small Table 4: coefficients $d_{ij}$ for $E_8$.}
  \vskip10pt
A  full list of the invariants  and the code to compute them is available in  \cite{CI}.

\vskip5pt
\footnotesize{
\vskip20pt
Dipartimento di Matematica, Sapienza Universit\`a di Roma, P.le A. Moro 2,
00185, Roma, Italy; 
\par\noindent
Email addresses:
\par\noindent{\tt deconcin@mat.uniroma1.it}
\par\noindent{\tt papi@mat.uniroma1.it}
\par\noindent{\tt procesi@mat.uniroma1.it}
}

\begin{thebibliography}{100}
\bibitem{B} Y.~Bazlov,  {\em Graded Multiplicities in the Exterior Algebra}, Adv. Math. {\bf 158}, 129--153 (2001)
\bibitem{Bou} N.~Bourbaki,
{  Groupes et alg\`ebres de Lie}
Hermann, Paris, 1968
\bibitem{bps} M. ~Bresar, C~Procesi, S. ~Spenko, {\em Functional identities on matrices and the Cayley--Hamilton polynomial,} arXiv:1212.4597, to appear in {\it Advances in Mathematics}  
\bibitem{Car} 
 H.~ Cartan,   {\em La transgression dans un groupe de Lie et dans un espace fibr\'e principal. } Colloque de topologie (espaces fibr\'es), Bruxelles, 1950, pp. 57Ð71. Georges Thone, Li\`ege; Masson et Cie., Paris, 1951.
\bibitem{Chev} C.~Chevalley, {\em  The Betti numbers of the exceptional Lie groups}, in ``Proc. International
Congress of Mathematicians'', 1950,'' Vol. II, pp. 21--24.
\bibitem{CI} C~De Concini, P~Papi, C~Procesi,  Invariants of $E_8$, http://www1.mat.uniroma1.it/people/papi/E8.zip
\bibitem{DCPPM} C~De Concini, P~M\"oseneder Frajria, P~Papi, C~Procesi,  
{\em On special covariants in the exterior algebra of a simple Lie algebra}, Rend. Lincei Mat. Appl. \textbf{25} (2014), 331--334.
\bibitem{Dolce}S.~Dolce, {\em   On certain modules of covariants in exterior algebras}, to appear in {\it Algebras and Representation Theory}.
\bibitem{KM}  S.~Katz, D.~Morrison, {\em Gorenstein threefold singularities with small resolutions via invariant theory for Weyl groups.} J. Algebraic Geom. 1 (1992), no. 3, 449Ð530
\bibitem{Given} A.B. Givental', \emph{Convolution of invariants of groups generated by reflections that are associated to simple singularities of functions.} (Russian)
Funktsional. Anal. i Prilozhen. \textbf{14} (1980), no. 2, 4--14
\bibitem{Kostant} B.~Kostant, {\em Eigenvalues of a Laplacian and commutative Lie subalgebras},
Topology,
{\bf 3} (1965), 147--159.
\bibitem{K}
B.~Kostant, {\em Clifford algebra analogue of the Hopf-Koszul-Samelson theorem, the
$\rho$-decomposition  $C(\g)=End\,V_\rho\otimes C(P)$, and the $\g$-module structure of $\bigwedge \g$}, Adv. Math.
{\bf 125} (1997), 275--350.
\bibitem{Ko} J. L. Koszul, {\em Homologie et cohomologie des alg\`ebres  de Lie}, Bull. Soc. Math. Fr. \textbf{78}
(1950), 65--127.
\bibitem{L}C.~Y~Lee,{\em  \ 
 Invariant polynomials of Weyl groups and applications to the centres of universal enveloping algebras.}
Canad. J. Math. 26 (1974), 583Ð592. 
\bibitem{Mein} E.~Meinrenken, {\em
Clifford algebras and Lie theory.}
Ergebnisse der Mathematik und ihrer Grenzgebiete. 3. Folge. A Series of Modern Surveys in Mathematics, 58. Springer, Heidelberg, 2013
\bibitem{P} C.~ Procesi {\em
On the theorem of Amitsur--Levitzki,} arXiv:1308.2421 (to appear in Israel Journal of Mathematics), doi: 10.1007/s11856-014-1118-8
\bibitem{R}
M.~Reeder, {\em Exterior powers of adjoint representation}, Canad. J. Math. {\bf 49} (1997),
133-159.
\bibitem{Saito}ÊK.~Saito,  T.~Yano, J.~Sekiguchi,{\em
On a certain generator system of the ring of invariants of a finite reflection group.} 
Comm. Algebra 8 (1980), no. 4, 373--408. 
\bibitem{Sl} P.~Slodowy,
    {\em Simple singularities and simple algebraic groups},
   {Lecture Notes in Mathematics},
 {815},
{Springer},
  {Berlin},
   {1980},
{x+175}



\end{thebibliography}
\end{document}